\RequirePackage{fix-cm}
\documentclass[smallextended]{svjour3}       
\smartqed  
\usepackage{graphicx}
\usepackage[titletoc]{appendix}

\usepackage{mathrsfs,amsmath,amssymb,bm,subfigure,float,geometry}
\usepackage{hyperref}
\usepackage{booktabs,mathtools}
\usepackage{algorithm}
\usepackage{lipsum}
\usepackage{amsfonts}
\usepackage{graphicx}
\usepackage{epstopdf}
\usepackage{makecell, rotating}
\usepackage{multirow}
\usepackage{algorithmic}
\usepackage{ntheorem}
\usepackage{nicematrix}
\usepackage{geometry}
\usepackage{nicematrix}
\usepackage{tikz}
\usepackage{dashbox}
\usetikzlibrary{calc}

\usepackage{indentfirst}
\usepackage[ulem={normalem,normalbf}]{changes}

\newcommand{\bbR}{\mathbb{R}}
\newcommand{\bbC}{\mathbb{C}}
\newcommand{\bbZ}{\mathbb{Z}}

\newcommand{\bbT}{\mathbb{T}}

\newcommand\tbbint{{-\mkern -16mu\int}}

\newcommand\dbbint{{-\mkern -19mu\int}}

\newcommand\bbint{
	{\mathchoice{\dbbint}{\tbbint}{\tbbint}{\tbbint}}
}

\allowdisplaybreaks[4]

\begin{document}

\title{Convergence analysis of PM-BDF2 method for quasiperiodic parabolic equations

}


\author{Kai Jiang 
\and Meng Li 
\and Juan Zhang
\and Lei Zhang
}


\institute{Hunan Key Laboratory for Computation and Simulation in Science and Engineering, Key Laboratory of Intelligent Computing
and Information Processing of Ministry of Education, Department of Mathematics and Computational Science, Xiangtan University, Xiangtan, Hunan, 411105, P. R. China. \at 
\email{kaijiang@xtu.edu.cn (Kai Jiang),
    limeng@smail.xtu.edu.cn (Meng 
    Li),
    zhangjuan@xtu.edu.cn (Juan Zhang).}\at
School of Mathematical Sciences, Institute of Natural Sciences, MOE-LSC, Shanghai Jiao Tong University, Shanghai, 200240, China.\at
\email{lzhang2012@sjtu.edu.cn (Lei Zhang).
              }         
}

\date{Received: date / Accepted: date}

\maketitle

\begin{abstract}
Numerically solving parabolic equations with quasiperiodic coefficients is a significant challenge due to the potential formation of space-filling quasiperiodic structures that lack translational symmetry or decay. In this paper, we introduce a highly accurate numerical method for solving time-dependent quasiperiodic parabolic equations. We discretize the spatial variables using the projection method (PM) and the time variable with the second-order backward differentiation formula (BDF2). We provide a  complexity analysis for the resulting PM-BDF2 method. Furthermore, we conduct a detailed convergence analysis, demonstrating that the proposed method exhibits spectral accuracy in space and second-order accuracy in time. Numerical results in both one and two dimensions validate these convergence results, highlighting the PM-BDF2 method as a highly efficient algorithm for addressing quasiperiodic parabolic equations.
\keywords{Quasiperiodic parabolic equation \and 
Projection method\and
Second-order backward differentiation formula\and
Convergence analysis.}
\subclass{  65D05 \and 65D15 \and 35B15 \and 68Q25}
\end{abstract}

\section{Introduction}
In this paper, we consider the following  quasiperiodic parabolic equation (QPE)
\begin{equation}\label{eqn:parabolic}
\left\{
		\begin{aligned}
			&\frac{\partial u(\bm{x}, t)}{\partial t}+\mathcal{L}u(\bm{x}, t)=f(\bm{x}, t), \quad(\bm{x}, t)\in R_T,\\
			&u(\bm{x}, 0)=u^0(\bm{x}),
		\end{aligned}    			
		\right.
\end{equation}
 where the set $R_T:=\bbR^d\times (0, T)$ for the fixed time $T>0$. The second order quasiperiodic elliptic operator $\mathcal{L}$ is defined as  $\mathcal{L}u(\bm{x}, t):=-\mathrm{div}(\alpha(\bm{x})\nabla u(\bm{x}, t))$, with $\alpha(\bm{x})$ and $f(\bm{x}, t)$ being quasiperiodic coefficient and source term in the spatial direction, respectively.

Parabolic equations typically refer to partial differential equations (PDEs) that describe heat conduction and diffusion phenomena. These equations are fundamental in many areas of science and engineering.  
{They have been extensively studied in various contexts, including numerical analysis and adaptive methods for solving them. In particular, adaptive finite element methods, which adjust mesh sizes in both space and time based on the problem's characteristics, have proven to be highly effective. These methods, as discussed in works by Douglas \cite{douglas1956numerical} and Eriksson \cite{eriksson1991adaptive}, enable more efficient and accurate solutions, particularly when dealing with complex or evolving problems.}
The study of parabolic equations with periodic coefficients has reached a relatively mature stage, with significant contributions from researchers like Amann \cite{amann1978periodic}, Mercier \cite{mercier1989introduction}, and Zhikov \cite{zhikov2006estimates}. These studies have provided a solid foundation for understanding the behavior of solutions in periodic settings.
In recent decades, there has been increasing interest in quasiperiodic parabolic equations (QPEs) due to their ability to model complex phenomena such as turbulent flows, as highlighted by researchers like Maestrello \cite{maestrello1979quasi} and Takeda \cite{takeda1999quasi}. This growing interest is driven by the fascinating and intricate behaviors exhibited by systems under quasiperiodic influences.
Numerous mathematical studies have been dedicated to exploring QPEs. Zaidman \cite{zaidman1961soluzioni} was one of the pioneers in this field, establishing the existence of quasiperiodic solutions for linear parabolic equations. Following his work, several researchers, including Nakao \cite{nakao1976bounded,nakao1978bounded}, Ward \cite{ward1988bounded}, and Yoshizawa \cite{yoshizawa2012stability}, have extended these results to semilinear and nonlinear QPEs, demonstrating the robustness of quasiperiodic solutions in more complex settings.
A notable recent contribution by Riccardo Montalto \cite{montalto2021navier} further advanced the understanding of QPEs. Montalto's work demonstrated the existence of quasiperiodic solutions for a general nonlinear parabolic equation defined on a torus \(\mathbb{T}^d\), with \(d \geq 2\), under the influence of a time-quasiperiodic external force. This result represents a significant step forward in the study of nonlinear QPEs, providing new insights into their behavior and potential applications.
 

However, numerically solving  QPEsemains a significant challenge due to their solutions
being quasiperiodic, exhibiting globally ordered structures without translational symmetry or decay. A common numerical method to solve QPEs 
  is the finite element method, which restricts the global PDEs \eqref{eqn:parabolic} to a finite domain. Although there exist some adaptive ways to improve convergence  \cite{braack2011duality,bastidas2021numerical},  it remains difficult to overcome the influence of irrational numbers, as discussed in  \cite{jiang2023approximation}. Therefore, there is still a lack of highly precise and efficient numerical
algorithms for solving QPE.  
Recently, the projection method (PM) \cite{jiang2014numerical,jiang2018numerical} has emerged as a highly accurate and efficient approach for approximating quasiperiodic systems.   Extensive studies have demonstrated that the PM can achieve high accuracy in computing various quasiperiodic
 systems \cite{cao2021computing,jiang2015stability,jiang2022tilt,xueyang2021numerical}. The comprehensive function approximation analysis of PM has been presented in \cite{jiang2024numerical}. Furthermore, PM has been effectively applied to solving quasiperiodic elliptic equations \cite{jiang2024projection}, with comprehensive numerical analysis conducted. It has also successfully addressed quasiperiodic homogenization in multiscale quasiperiodic elliptic equations.    These advancements provide valuable insights into our research and applications.
The purpose of this paper is to propose efficient methods for accurately solving spatially quasiperiodic solutions of  quasiperiodic parabolic equations, and to establish the corresponding convergence analysis. Concretely, we employ the PM to discretize the  quasiperiodic solution of QPE \eqref{eqn:parabolic} and derive corresponding  discrete scheme. Furthermore, we present a rigorous error analysis to demonstrate the 
spectral
accuracy of the PM. 


\textbf{Organization:} The paper is organized as follows. Section \ref{sec:quasiperiodic_function} introduces the quasiperiodic function spaces. Section \ref{sec:quasiperiodic_parabolic} discusses the well-posedness of QPE \eqref{eqn:parabolic}. Section \ref{sec:numerical_methods} introduces the combination of PM with the second-order backward differentiation formula (BDF2) to discretize \eqref{eqn:parabolic}, leading to the PM-BDF2 method. We also provide its numerical implementation and an analysis of its computational complexity.  Section \ref{sec:convergence_analysis} offers the convergence analysis of the proposed method. Section \ref{sec:numerical_experiments} presents numerical results that further validate this analysis. Finally, the conclusions of the paper are summarized in Section \ref{sec:conclusion}.
 
 \section{Quasiperiodic function space}\label{sec:quasiperiodic_function}
In this section, we  present the definition and properties of  quasiperiodic function spaces.  For clarity, we will use the following notations: let $\Omega_L=[-L, L]^d\subset\bbR^d$ with volume $|\Omega_L|= (2L)^d$. We present the definition of quasiperiodic functions.
\begin{definition}
    A $d\times n$-order matrix $\bm{P}~(\textrm{with}~d\leq n)$ is the projection matrix, if it belongs to the set $\mathbb{P}^{d\times n }$ defined as $\mathbb{P}^{d\times n}:=\left\{\bm{P}=(\bm{p}_1,\cdots,\bm{p}_n)\in\mathbb{R}^{d\times n}:\bm{p}_1,\cdots,\bm{p}_n ~\mathrm{are}~\mathbb{Q}\text{-linearly independent}\right\}$.
\end{definition}

\begin{definition}
A $d$-dimensional function $u(\bm{x})$ is quasiperiodic, if there exists an $n$-dimensional periodic function $U$ and a projection matrix $\bm{P}\in\mathbb{P}^{d\times n}$, such that $u(\bm{x})=U(\bm{P}^T\bm{x})$ for all $\bm{x}\in\mathbb{R}^d$. $U$ is called the \emph{parent function} of $u(\bm{x})$. 
\end{definition}
 In particular, when $n = d$ and $\bm{P}$ is nonsingular, $u(\bm{x})$ is periodic. The continuous Fourier-Bohr transform of $u(\bm{x})$ is
\begin{equation}\label{eqn:continuous_ft}
    \hat{u}_{\bm{\lambda}}=\lim_{L\to \infty}\frac{1}{|\Omega_{L}|}\int_{\Omega_{L}} u(\bm{x})e^{-\imath\bm{\lambda}\cdot\bm{x}}d\bm{x}:=\bbint u(\bm{x})e^{-\imath\bm{\lambda}\cdot\bm{x}}d\bm{x},\quad \bm{\lambda}\in\bbR^d.
\end{equation}
The Fourier series associated to $u(\bm{x})$ is
\begin{align}
u(\bm{x})=\sum_{\bm{k}\in\mathbb{Z}^n}\hat{u}_{\bm{\lambda_k}} e^{\imath\bm{\lambda_k}\cdot\bm{x}},
\label{eqn:continuous_FS}
\end{align}
where $\bm{\lambda_k}\in\Lambda:=\left\{\bm{\lambda_k}=\bm{P}\bm{k}:\hat{u}_{\bm{\lambda_k}}\neq 0\right\}$ are Fourier frequencies and $\hat{u}_{\bm{\lambda_k}}$ are Fourier coefficients. The set $\Lambda$ is called  the \emph{spectral point set} of $u(\bm{x})$. Moreover, we have the Parseval's identity
\begin{equation}\label{eqn:parsevel}
    \sum_{\bm{k}\in\bbZ^n}\hat{u}_{\bm{\lambda_k}}=\bbint |u(\bm{x})|^2d\bm{x}.
\end{equation}

Correspondingly, we introduce the periodic and quasiperiodic function spaces.
 
\noindent$\bullet$ $L^2(\mathbb{T}^n)$ \textbf{space}: 
$L^2(\mathbb{T}^n)=\left\{U(\bm{y}):\displaystyle\frac{1}{|\mathbb{T}^n|}\displaystyle\int_{\mathbb{T}^n}|U|^2d\bm{y}<\infty\right\}$ equipped with the inner product 
$$
(U_1, U_2)_{L^2(\mathbb{T}^n)}=\frac{1}{|\mathbb{T}^n|}\int_{\mathbb{T}^n}U_1\overline{U}_2d\bm{y}.
$$

\noindent$\bullet$ $H^s(\mathbb{T}^n)$ \textbf{space}: for any integer $s\geq 0$, the $s$-derivative Hilbert space on $\bbT^n$ is
$$
H^s(\mathbb{T}^n)=\left\{U\in L^2(\mathbb{T}^n):\|U\|_{H^s(\mathbb{T}^n)}<\infty\right\},
$$
where 
$$\|U\|_{H^s(\mathbb{T}^n)}= \Bigg(\sum_{\bm{k}\in\mathbb{Z}^n}(1+|\bm{k}|^2)^s|\hat{U}_{\bm{k}}|^2\Bigg)^{1/2},\quad \hat{U}_{\bm{k}}=(U, e^{\imath\bm{k}\cdot\bm{y}})_{L^2(\mathbb{T}^n)}.$$ The semi-norm of $H^s(\mathbb{T}^n)$ is defined as  $|U|_{H^s(\mathbb{T}^n)}=\Big(\sum\limits_{\bm{k}\in\mathbb{Z}^n}|\bm{k}|^{2s}|\hat{{U}}_{\bm{k}}|^2\Big)^{1/2}.$

\noindent$\bullet$ $\overline{H}^s(\mathbb{T}^n)$ \textbf{space}: for any integer $s\geq 0$,  $$\overline{H}^s(\mathbb{T}^n):=\left\{U(\bm{y})\in {H}^s(\mathbb{T}^n): \displaystyle\frac{1}{|\mathbb{T}^n|}\displaystyle\int_{\mathbb{T}^n} U(\bm{y})d\bm{y}=0\right\}$$ equipped with ${H}^s(\mathbb{T}^n)$-norm.

\noindent$\bullet$ $\mathrm{QP}(\mathbb{R}^d)$ \textbf{space}: 
 $\mbox{QP}(\mathbb{R}^d)$ represents the space of all $d$-dimensional quasiperiodic functions.

\noindent$\bullet$ $L^q_{QP}(\mathbb{R}^d)$ \textbf{space}: for any fixed integer  $q\in[0,+\infty),$ denote
$$
L_{QP}^q(\mathbb{R}^d)=\left\{u(\bm{x})\in\mbox{QP}(\mathbb{R}^d):\|u\|_{L_{QP}^q}^q(\bbR^d)= \bbint|u(\bm{x})|^q d\bm{x}<\infty\right\}
$$ 
and
$$
L_{QP}^{\infty}(\mathbb{R}^d)=\left\{u(\bm{x})\in\mbox{QP}(\mathbb{R}^d):\|u\|_{L_{QP}^\infty}(\bbR^d)= \sup_{\bm{x}\in\mathbb{R}^d}|u(\bm{x})| <\infty\right\}.
$$
The inner product $(\cdot,\cdot)_{L_{QP}^2(\mathbb{R}^d)}$ is defined by 
$$
(u,v)_{L_{QP}^2(\mathbb{R}^d)}= \bbint u(\bm{x})\overline{v}(\bm{x})d\bm{x}.
$$
By Parseval's identity \eqref{eqn:parsevel}, we have
$$
\|u\|^2_{{L_{QP}^2(\mathbb{R}^d)}}=\sum_{\bm{\lambda}\in\Lambda}|\hat{u}_{\bm{\lambda}}|^2.
$$


\noindent$\bullet$ $H^s_{QP}(\mathbb{R}^d)$ \textbf{space}: for any integer $s\geq 0,$ the Sobolev space $H_{QP}^{s}(\mathbb{R}^d)$ comprises all quasiperiodic functions  with partial derivatives order $s$. The inner product $(\cdot,\cdot)_{H_{QP}^{s}(\mathbb{R}^d)}$ is 
$$
(u,v)_{H_{QP}^{s}(\mathbb{R}^d)}=(u,v)_{L_{QP}^{2}(\mathbb{R}^d)}+\sum_{|p|\leq s}(\partial_{\bm{x}}^p u,\partial_{\bm{x}}^p v)_{L_{QP}^{2}(\mathbb{R}^d)},\quad u,v\in H_{QP}^{s}(\mathbb{R}^d)
$$
and equipped with norm
$$
\|u\|^2_{H_{QP}^{s}(\mathbb{R}^d)}=\sum_{\bm{\lambda}\in\Lambda}(1+|\bm{\lambda}|^2)^s|\hat{u}_{\bm{\lambda}}|^2.
$$

\noindent$\bullet$ $\overline{H}^s_{QP}(\mathbb{R}^d)$ \textbf{space}: for any integer $s\geq 0$,  $\overline{H}^s_{QP}(\mathbb{R}^d):=\left\{u(\bm{x})\in H_{QP}^{s}(\mathbb{R}^d): \displaystyle\bbint u(\bm{x})d\bm{x}=0\right\}$ equipped with $H_{QP}^{s}(\mathbb{R}^d)$-norm.

\noindent$\bullet$ $H^{-1}_{QP}(\mathbb{R}^d)$ \textbf{space}: $H^{-1}_{QP}(\mathbb{R}^d)$ represents the dual space of $\overline{H}^{1}_{QP}(\bbR^d)$. The  $H^{-1}_{QP}(\bbR^d)$-norm is defined as
 $$
 \|u\|_{H^{-1}_{QP}(\bbR^d)}=\sup\left\{\langle u, v\rangle\Big|~ u\in{H}_{QP}^{-1}(\bbR^d),~  v\in \overline{H}^1_{QP}(\bbR^d),~ \|v\|_{H^1_{QP}(\bbR^d)}\leq 1\right\},
$$
where $\langle \cdot,\cdot\rangle$ is the pairing between $H^{-1}_{QP}(\bbR^d)$ and $\overline{H}^{1}_{QP}(\bbR^d)$. 

\noindent$\bullet$ $L^2(0, T; {H}^s_{QP}(\mathbb{R}^d))$ \textbf{space}: for any integer $s\geq 0$, the Sobolev space involving time $L^2(0, T; {H}^s_{QP}(\mathbb{R}^d))$ comprises functions mapping time into quasiperiodic Hilbert space. This space is equipped with the norm as follows
$$
\|u\|_{L^2(0, T; {H}^s_{QP}(\mathbb{R}^d))}:=\left(\int_{0}^T\|u(t)\|^2_{{H}^s_{QP}(\mathbb{R}^d)}dt\right)^{1/2}.
$$

For the case of $s=0$, we have $H_{QP}^0(\mathbb{R}^d)=L^2_{QP}(\mathbb{R}^d)$.
To simplify notation, we use $\|\cdot \|_0$,   $\|\cdot\|_s$,  and $\|\cdot\|_{{s},T}$ to denote the norm$\|\cdot\|_{L^2_{QP}(\mathbb{R}^d)}$, $\|\cdot\|_{H_{QP}^{s}(\mathbb{R}^d)}$ and $\|\cdot\|_{L^2(0, T; {H}^s_{QP}(\mathbb{R}^d))}$, respectively.


\section{Quasiperiodic parabolic equation}\label{sec:quasiperiodic_parabolic}
Building on the preparation from the previous section, we now discuss the well-posedness of QPE \eqref{eqn:parabolic}. The quasiperiodic coefficient
 $\alpha(\bm{x})\in L^{\infty}_{QP}(\mathbb{R}^d)$ is uniformly elliptic and bounded, i.e., $\forall\bm{x}, \xi\in\bbR^d$,
\begin{equation}\label{eqn:ellipticity}
\gamma_0|\xi|^2\leq\xi^T\alpha(\bm{x})\xi\leq\gamma_1|\xi|^2,\quad \gamma_0, \gamma_1>0    
\end{equation}
holds. The source term $f(\bm{x}, t)\in L^2(R_T)$ and initial value $u^0(\bm{x})\in L^2_{QP}(\mathbb{R}^d)$.

By multiplying \eqref{eqn:parabolic} with an arbitrary test function $v\in \overline{H}^1_{QP}(\bbR^d)$ and integrating over the $\bbR^d$, we have
 \begin{equation}
 (u', v)+ B[u,v;t]= (f,v) \quad\left('=\frac{d}{dt}\right)
 \end{equation}
for each $0\leq t\leq T$, where $B[u, v ;t]$ denotes the time-dependent bilinear form defined as 
$$
B[u,v;t]:=\bbint\alpha\nabla u(\cdot, t)\nabla vd\bm{x}\quad u\in L^2(0, T; \overline{H}^1_{QP}(\bbR^d)), ~v\in \overline{H}^1_{QP}(\bbR^d).
$$
The associated variational formulation of \eqref{eqn:parabolic} is to seek a $u\in L^2(0, T; \overline{H}^1_{QP}(\bbR^d))$ with $u'\in L^2(0, T; {H}^{-1}_{QP}(\bbR^d))$  such that
 \begin{equation}\label{eqn:variational}
 \left\{\begin{aligned}
     &(u', v)+B[u,v;t]=(f, v)\quad \forall v\in \overline{H}^1_{QP}(\bbR^d),\\
     &u(\cdot,0)=u^0.
 \end{aligned}
 \right.
 \end{equation}
 
 The variational problem \eqref{eqn:variational} is well-posed, as stated in the following theorem.
 \begin{theorem}\label{thm:wellposed}
   There exists a unique solution of the problem \eqref{eqn:variational}.
\end{theorem}
The proof of this result relies on an energy estimate in the quasiperiodic case, which is given in the following theorem.

 \begin{theorem}[Energy estimate]\label{lemma:energy_estimate}
 There exits a constant $C$ such that
 $$
 \max_{0\leq t\leq T}\|u_m(t)\|_{0}+\|u_m\|_{1,T} + \|u'_m\|_{-1,T}\leq C\|f\|_{0,T}+\|u^0\|_0,
 $$
 where 
 $$u_m(t):=\sum_{j=1}^m\hat{u}_m^j(t)e^{\imath\bm{\lambda}_j\cdot\bm{x}}\in\overline{H}^1_{QP}(\mathbb{R}^d)$$
 for each $0\leq t\leq T$.
 \end{theorem}
 \begin{proof}
 The proof of this theorem follows the classical approach, with the key difference being that the quasiperiodic function space is defined over the entire space $\mathbb{R}^d$. Therefore, it suffices to verify whether the following inequality holds: for $0\leq t\leq T$,
 \begin{equation}\label{eqn:energy_ellptic}
     \gamma_0\|u_m\|^2_1\leq B[u_m, u_m; t]+\gamma\|u_m\|^2_0,~\gamma_0,\gamma>0.
 \end{equation}
  As shown in \cite{jiang2024projection}, we have proved the Poincare's inequality in the quasiperiodic case, which states that 
 $$
 \|u\|_1\leq C|u|_1,~~\forall u\in\overline{H}^1_{QP}(\mathbb{R}^d).
 $$
This result guarantees that \eqref{eqn:energy_ellptic} holds. The remainder of the proof follows the standard approach, and further details can be found in \cite[pp.377]{evans2022partial}.
 \end{proof}

 Based on the  energy estimate established in Theorem \ref{lemma:energy_estimate},
we can follow a similar strategy to that in \cite[pp.378-- pp.379]{evans2022partial} prove the well-posedness of the variational problem \eqref{eqn:variational}.

\section{Numerical methods}\label{sec:numerical_methods}
In this section, we introduce the PM-BDF2 method for the QPE \eqref{eqn:parabolic}, which employs PM to discretize in the spatial direction and utilizes the second-order backward differentiation formula (BDF2) to discretize in the temporal direction. Furthermore, we present the numerical implementation for solving the QPE \eqref{eqn:parabolic} and provide an analysis of the corresponding computational complexity.

\subsection{PM and BDF2 scheme}
In this subsection, we combine the PM  with the BDF2 scheme for efficient and stable numerical solutions to QPE \eqref{eqn:parabolic}. 
The PM embedded a quasiperiodic function 
$u(\bm{x})$ into a high-dimensional periodic function $U(\bm{y})$. {Specifically, the PM allows us to  compute the  Fourier coefficients  of the periodic parent function, leveraging the efficiency of the FFT algorithm. The quasiperiodic structure is then obtained by projecting the high-dimensional periodic structure onto the irrational manifold.} 

For an integer $N\in\mathbb{N}^{+}$ and a given projection matrix $\bm{P}\in\mathbb{P}^{d\times n}$, we denote	 
\begin{align*}
	K_N^{n}=\left\{\bm{k}=(k_j)_{j=1}^{n}\in\mathbb{Z}^{n}:-N/2\leq k_j <N/2 \right\}.
\end{align*}
Next, we discretize the torus $\mathbb{T}^{n}$ and denote the discrete tours as
$$
	\mathbb{T}_N^n=\left\{\bm{y_j}=( 2\pi j_1/N, \cdots, 2\pi j_n/N )\in\mathbb{T}^{n}: 0\leq j_1,\cdots, j_n\leq N-1, ~\bm{j}=(j_1,\cdots, j_n)\right\}.
$$
This implies that we set $N$ discrete points uniformly in each spatial direction of the $n$-dimensional space. Consequently, there are $D=N^n$ discrete points in total. Let $\mathcal{G}_N$ denote the grid function space
	$$
	\mathcal{G}_N=\left\{U:\mathbb{Z}^{n}\mapsto\mathbb{C},~ U\text{ is a periodic function on} ~
     \mathbb{T}_N^n\right\}
	$$
	and define the $\ell^2$-inner product $\left(\cdot , \cdot \right)_N$ on $\mathcal{G}_N$ as
	$$
	\left(\phi_1,\phi_2\right)_N=\frac{1}{D}\sum_{\bm{y_j}\in\mathbb{T}^n_N}\phi_1(\bm{y_j})\overline{\phi}_2(\bm{y_j}),~~\phi_1,\phi_2\in \mathcal{G}_N.
	$$
For $\bm{k}_1, \bm{k}_2\in\mathbb{Z}^n$, $\bm{P}\in\mathbb{P}^{d\times n}$, we denote $\bm{Pk}_{1}\overset{N}{=}\bm{Pk}_2$, if $\bm{k}_1, \bm{k}_2$ satisfy 
$$
\bm{Pk}_1=\bm{Pk}_2 +N\bm{Pm},\quad N\in\mathbb{Z},\quad\bm{m}\in\mathbb{Z}^n.
$$ 
Then we have the following discrete orthogonality
\begin{equation}\label{eq:orth}
\Big(e^{\imath\bm{k}_1\cdot\bm{y_j}},e^{\imath\bm{k}_2\cdot\bm{y_j}}\Big)_N=
		\left\{
		\begin{aligned}
			&1, ~~~~\bm{k}_1\overset{N}{=}\bm{k}_2,\\
		  &0,~~~~ \mbox{otherwise},\\
		\end{aligned}    			
		\right.
	\end{equation}
here, $\bm{P}$ is an identity matrix of $n$-order. Using \eqref{eq:orth}, we can calculate the discrete Fourier coefficients of a high-dimensional periodic function 
 $U(\bm{y})$ as follows 
$$
\widetilde{U}_{\bm{k}}=\Big(U(\bm{y_j}),e^{\imath\bm{k}\cdot\bm{y_j}}\Big)_N=\frac{1}{D}\sum_{\bm{y_j}\in\mathbb{T}^n_N}U(\bm{y_j})e^{-\imath\bm{k}\cdot\bm{y_j}}.
$$
{The following lemma, based on the Birkhoff’s ergodic theorem \cite{pitt1942some}, clarifies  that the Fourier-Bohr transform of the quasiperiodic function $u$ is equivalent to the Fourier transform of its parent function $U$. The detailed proof of this result can refer to  \cite[Theorem 4.1]{jiang2024numerical}.
\begin{lemma}
    For a given quasiperiodic function $$u(\bm{x})=U(\bm{P}^T\bm{x}),~\bm{x}\in\bbR^d,$$
    where $U$ is  its parent function defined on the tours $\bbT^n$ and $\bm{P}$ is the projection matrix, we have
    $$
    \hat{u}_{\bm{\lambda_k}}=\hat{U}_{\bm{k}},~~\bm{\lambda_k}=\bm{Pk}.
    $$
\end{lemma}}

{Consequently, the PM assigns $\widetilde{U}_{\bm{k}}$ as the Fourier-Bohr coefficient $\widetilde{u}_{\bm{\lambda_k}}, ~\bm{\lambda_k}=\bm{Pk}$.} Therefore, we can define the \emph{discrete Fourier-Bohr transform} of quasiperiodic function $u(\bm{x})$ in form
\begin{equation}\label{eqn:dft}  u(\bm{x_j})=\sum_{\bm{\lambda_k}\in\Lambda_{N}^{d}}\widetilde{u}_{\bm{\lambda_k}}e^{\imath\bm{\lambda_k}\cdot\bm{x_j}},\quad \Lambda_N^d=\left\{\bm{\lambda_k}: \bm{\lambda_k}=\bm{Pk}, ~\bm{k}\in K_N^n\right\},
\end{equation}
where $\bm{x_j}\in X_{\bm{P}}:=\left\{\bm{x_j}=\bm{Py_j}: ~\bm{y_j}\in\mathbb{T}_N^n,~ \bm{P}\in\mathbb{P}^{d\times n}\right\}$ are \emph{collocation points}. 

The truncation of quasiperiodic function $u(\bm{x})$ can be expressed as   
\begin{equation}\label{eqn:truncation}
    T_N u(\bm{x})=\sum_{\bm{\lambda_k}\in\Lambda_{N}^d}\hat{u}_{\bm{\lambda_k}} e^{\imath\bm{\lambda_k}\cdot\bm{x}}:=u_T(\bm{x}),
\end{equation}
 and the trigonometric interpolation of
$u(\bm{x})$ is
\begin{equation}\label{eqn:interpolation} I_N u(\bm{x})=\sum_{\bm{\lambda_k}\in\Lambda_{N}^{d}}\widetilde{u}_{\bm{\lambda_k}}e^{\imath\bm{\lambda_k}\cdot\bm{x}}:=u_N(\bm{x}).
\end{equation}
 Consequently, $u(\bm{x_j}) = I_Nu(\bm{x_j})$. From a computational perspective, it is apparent that the PM
 can use the multidimensional FFT algorithm to obtain the Fourier coefficients of a quasiperiodic
 function.

 
The BDF2 scheme is a common time-discrete method for solving parabolic equations in the time direction. It is particularly known for its high stability. For the time grid points $0 = t_0 < t_1 < \cdots < t_M = T$, where $t_m = m\tau, m = 0,1,...,M$, and the time step size $\tau = T/M$, the BDF2  numerically approximates the solution at each time step through the following relation:
\begin{equation}{\label{eqn:bdf2}}
    (3-2\tau \mathcal{L})u^{m}\approx 4 u^{m-1}- u^{m-2}+2\tau  f^{m},\quad 1\leq m\leq M.
\end{equation}

\subsection{PM-BDF2 discretization}
In this subsection, we use the  PM-BDF2 to discretize the QPE \eqref{eqn:parabolic}, outlining the two steps involved in the process.

 $\bullet$ \textbf{PM-BDF2-Step 1:} From $t_0$ to $t_1$, we employ the PM to discretize \eqref{eqn:parabolic} in space direction at $t=t_0$:
$$
\alpha_N=\alpha(\bm{x_j})= \sum_{\bm{\lambda}_\alpha\in \Lambda_{\alpha, N}^d}\widetilde{\alpha}_{\bm{\lambda}_\alpha}e^{\imath\bm{\lambda}_\alpha\cdot\bm{x_j}}, \quad \widetilde{\alpha}_{\bm{\lambda}_\alpha}=\widetilde{A}_{\bm{k}_A}=\frac{1}{D}\sum_{\bm{k}_A\in\bbZ^n}A(\bm{y_j})e^{-\imath\bm{k}_A\cdot\bm{y_j}},
$$
$$
u_N^0:=u(\bm{x_j}, t_0)=\sum_{\bm{\lambda}_u\in \Lambda_{u, N}^d}\widetilde{u}_{\bm{\lambda}_u}(t_0)e^{\imath\bm{\lambda}_u\cdot\bm{x_j}},\quad \widetilde{u}_{\bm{\lambda}_u}(t_0)=\widetilde{U}_{\bm{k}_U}(t_0)=\frac{1}{D}\sum_{\bm{k}_U\in\bbZ^n}U(\bm{y_j}, t_0)e^{-\imath\bm{k}_U\cdot\bm{y_j}},
$$
$$
f_N^0:=f(\bm{x_j}, t_0)=\sum_{\bm{\lambda}_f\in \Lambda_{f, N}^d}\widetilde{f}_{\bm{\lambda}_f}(t_0)e^{\imath\bm{\lambda}_f\cdot\bm{x_j}},\quad \widetilde{f}_{\bm{\lambda}_f}(t_0)=\widetilde{F}_{\bm{k}_F}(t_0)=\frac{1}{D}\sum_{\bm{k}_F\in\bbZ^n}F(\bm{y_j}, t_0)e^{-\imath\bm{k}_F\cdot\bm{y_j}},
$$
where $\bm{x_j}\in X_{\bm{P}}$ and $A, U, F$ are parent functions of $\alpha, u, f$, respectively. 
 
 {In the first time step, we use the first-order backward differentiation formula (BDF1) because only the current time point and the initial condition are available.} Next, we define the finite dimensional linear
 subspace of $\mathrm{QP}(\mathbb{R}^d)$ as $V^N:=\mathrm{span}\big\{e^{\imath\bm{\lambda}_v\cdot\bm{x}}:\bm{\lambda}_v\in\Lambda_{u, N}^d,\bm{x}\in\mathbb{R}^{d}\}.$ We then select a test function $v$ from the space $\overline{V}^N:=\left\{ v\in V^N:~ \widetilde{v}_{\bm{0}}=0\right\}\subseteq\overline{{H}}_{QP}^1(\mathbb{R}^d)$.   Our goal in the following is to find an approximate solution of \eqref{eqn:variational} in $\overline{V}^N$, i.e., $u_N^0\in \overline{V}^N$, such  that
\begin{equation}
    \left(\frac{u_N^1-u_N^0}{\tau}, v\right)_N+(\alpha_N\nabla u_N^0, \nabla v)_N\approx(f_N^0, v)_N \quad \forall v\in\overline{V}^N,
\end{equation}
 Combining with the orthogonality of the basis function, we obtain the  scheme
\begin{equation}\label{eqn:pm_bdf2_1_discrete}
   \widetilde{U}_{\bm{k}_V}(t_1)\approx\widetilde{U}_{\bm{k}_V}(t_0) -\tau\sum_{\bm{k}_U\in K_N^n}\widetilde{A}_{\bm{k}_V\overset{N}{-}\bm{k}_U}(\bm{Pk}_V)^T(\bm{Pk}_U)\widetilde{U}_{\bm{k}_V}(t_0)+\tau\widetilde{F}_{\bm{k}_V}(t_0),
\end{equation}
where $\bm{k}_{V}\overset{N}{-}\bm{k}_{U}:=(\bm{k}_{V}-\bm{k}_{U}) (\mathrm{mod}~N)$.

The discrete scheme \eqref{eqn:pm_bdf2_1_discrete} is formulated with respect to the tensor index $\bm{k}=(k_l)_{l=1}^n\in K_N^n$. To solve it numerically, we apply the tensor-vector-index conversion. 
The corresponding  \emph{convert bijection} is defined as
\begin{equation}
     \begin{aligned}
    \mathcal{C}: \quad&K_N^n \rightarrow \mathbb{N},\\
	&\bm{k}\xrightarrow{\mathcal{C}} i, 
     \end{aligned}
\end{equation}
 where $i$ is determined by the rule
\begin{equation}
    i=\sum_{l=1}^{n} \overline{k}_l N^{n-l},\quad \overline{k}_l:= k_l(\mathrm{mod}~N).
\end{equation}
The index inverse conversion $\mathcal{C}^{-1}$ is defined by
\begin{equation}	k_l=
		\left\{
		\begin{aligned}
			&\overline{k}_l, ~~~~0\leq\overline{k}_l< N/2, \\
		  &\overline{k}_l-N,~~~~ N/2\leq\overline{k}_l< N,\\
		\end{aligned}    			
		\right.
  \quad \overline{k}_{l}=\Big\lfloor i~ (\mathrm{mod}~N^{n-l+1})/ N^{n-l} \Big\rfloor,
	\end{equation}
where  $\lfloor \cdot\rfloor$ is the  floor symbol. 

By using the tensor-vector-index conversion $\mathcal{C}$,  we obtain
\begin{equation}\label{eqn:conversion}
    \bm{k}_V\xrightarrow{\mathcal{C}} i, \quad  \bm{k}_U\xrightarrow{\mathcal{C}} j,
\end{equation}
$$
B_{ij}=\widetilde{A}_{\bm{k}_V\overset{N}{-}\bm{k}_U}, \quad W_{ij}=(\bm{Pk}_V)^T(\bm{Pk}_U),
$$
$$
\mathcal{U}^1(i)=\widetilde{U}_{\bm{k}_V}(t_1),\quad \mathcal{U}^0(i)=\widetilde{U}_{\bm{k}_V}(t_0), \quad \mathcal{F}^0(i)=\widetilde{F}_{\bm{k}_V}(t_0).
$$
Then the fully discrete scheme in this step  can be derived as
\begin{equation}\label{eqn:pm_bdf2_1}
    \mathcal{U}^1\approx \mathcal{U}^0- \tau Q\mathcal{U}^0+\tau\mathcal{F}^0,
\end{equation}
where
$$
Q=B\circ W,    
$$
$$
Q, B, W\in\bbC^{D\times D}.
$$

It's worth mentioning that $B$ is a multi-level block circulant matrix, and its structure of $B$ can be depicted by 
\begin{scriptsize}
\begin{equation*}
      \begin{gathered}
                B_{11}^{(1)}\\
                \begin{bNiceMatrix}[name=a]
                    \widetilde{A}_{(\bm{0}, 0)}&\widetilde{A}_{(\bm{0}, -1)}& \cdots&\widetilde{A}_{(\bm{0},1)}\\
                    \widetilde{A}_{(\bm{0}, 1)}&	\widetilde{A}_{(\bm{0}, 0)}&\cdots &	\widetilde{A}_{(\bm{0}, 2)}\\
                    \vdots&\vdots&\cdots&\vdots\\
                    \widetilde{A}_{(\bm{0}, -1)}&	\widetilde{A}_{(\bm{0}, -2)}& \cdots&	\widetilde{A}_{(\bm{0}, 0)}
                \end{bNiceMatrix}
            \end{gathered}
            \begin{matrix}
                \\
            \longrightarrow
        \end{matrix}\ \ 
        \begin{gathered}
            B_{11}^{(2)}\\
            \begin{bNiceMatrix}[name=b]
                B_{11}^{(1)}&B_{12}^{(1)}& \cdots&B_{1N}^{(1)}\\
                B_{1N}^{(1)}&	\dbox{$B_{11}^{(1)}$}&\cdots &B_{1(N-1)}^{(1)}\\
      \vdots&\vdots&\cdots&\vdots\\
                B_{12}^{(1)}&	B_{13}^{(1)}&\cdots &B_{11}^{(1)}
            \end{bNiceMatrix}
            \end{gathered}
        \begin{matrix}
                \underset{\longrightarrow}{...}
            \end{matrix}
            \begin{gathered}
        B\\
        \begin{bNiceMatrix}[name=c]
            B_{11}^{(n-1)}&B_{12}^{(n-1)}&\cdots &B_{1N}^{(n-1)}\\
            B_{1N}^{(n-1)}&	\dbox{$B_{11}^{(n-1)}$}&\cdots &B_{1(N-1)}^{(n-1)}\\
            \vdots&\vdots&\cdots&\vdots\\
            B_{12}^{(n-1)}&	B_{13}^{(n-1)}&\cdots &B_{11}^{(n-1)}
		\end{bNiceMatrix},\end{gathered}
\end{equation*}
\end{scriptsize}
\tikz[remember picture, overlay]
{
\draw[cyan][->][color=black] ($(a-1-1)-(0.6,-0.1)$) to[out=180,in=180] ($(a-2-1)-(0.6,-0.1)$);
\draw[cyan][->][color=black] ($(a-2-1)-(0.60,0)$) to[out=180,in=180] ($(a-3-1)-(0.60,0)$);
\draw[cyan][->][color=black] ($(a-3-1)-(0.60,0)$) to[out=180,in=180] ($(a-4-1)-(0.60,0)$);
\draw[cyan][->][color=black] ($(b-1-1)-(0.40,0)$) to[out=180,in=180] ($(b-2-1)-(0.40,0)$);
\draw[cyan][->][color=black] ($(b-2-1)-(0.40,0)$) to[out=180,in=180] ($(b-3-1)-(0.40,0)$);
\draw[cyan][->][color=black] ($(b-3-1)-(0.40,0.1)$) to[out=180,in=180] ($(b-4-1)-(0.40,0)$);
\draw[cyan][->][color=black] ($(c-1-1)-(0.55,0)$) to[out=180,in=180] ($(c-2-1)-(0.55,0)$);
\draw[cyan][->][color=black] ($(c-2-1)-(0.55,0)$) to[out=180,in=180] ($(c-3-1)-(0.55,0)$);
\draw[cyan][->][color=black] ($(c-3-1)-(0.55,0.1)$) to[out=180,in=180] ($(c-4-1)-(0.55,0)$);
}
where $\bm{0}$ denotes an $(n-1)$-dimensional zero vector. 
To obtain the matrix $B$, we begin by computing the discrete Fourier coefficients of $A(\bm{x})$, denoted as $\widetilde{A}$. Each row of $\widetilde{A}$ is  cyclically permuted to assemble the first-level block circulant matrix $B^{(1)}$. This matrix is then recursively permuted  to construct the $l$-level block circulant matrix $B^{(l)}$ ($2\leq l\leq n$), which consists of $N\times N$ blocks of $B^{(l-1)}$. As a result, $B$ is referred to as an \emph{$n$-level block circulant matrix} and can be  expressed as follows
$$
B=\sum_{\bm{k}\in K_N^n}\widetilde{A}_{\bm{k}}Z_{N}^{\bm{k}}=\sum_{ k_1\in K_N^1}\cdots\sum_{k_n\in K_N^1}\widetilde{A}_{(k_1,\cdots,  k_n)}Z_{N}^{k_1}\otimes\cdots\otimes Z_{N}^{k_n},\quad \bm{k}=(k_1,\cdots, k_n),
$$
where
\begin{equation}\nonumber
\begin{array}{cl}
Z_N^k=\begin{array}{c@{\hspace{-6pt}}l}
\left[
\begin{array}{cccccc}
0&\cdots&1&0&\cdots&0\\
\vdots&&&\ddots&&\vdots\\
\vdots&&&&\ddots&\vdots\\
1&0&&&&1\\
&\ddots&&&&\\
0&\cdots&1&\cdots&\cdots&0\\
\end{array}
\right]_{N\times N}&\begin{array}{l}
\left.\rule{0mm}{10mm}\right\}\overline{k}\\
\\
\\
\\
\end{array}
\end{array}
\\
\hspace{7mm}\begin{array}{l}
\underbrace{\rule{9mm}{0mm}}_{N-\overline{k}}\quad\quad\quad\quad\qquad\quad\ \qquad
\end{array}
\end{array}
\end{equation}
is the \emph{cycling permutation matrix} and  $\otimes$ denotes the tensor product.
 With this structure, we can apply the  compressed storage method proposed by Jiang et al. in \cite{jiang2024projection} to efficiently reduce the memory and  accelerate the computation.

 $\bullet$ \textbf{PM-BDF2-Step 2:} For $2\leq m\leq M$, the discrete variational formulation is to  find $u_N^m\in\overline{V}^N$ such that 
 \begin{equation}\label{eqn:full_discrete}
     (\mathcal{D}^{\tau}u_N^m, v)_N + (\alpha_N\nabla u_N^m, \nabla v)_N \approx (f_N^m, v)_N, \quad \forall v\in \overline{V}^N,
 \end{equation}
 where 
 $$\mathcal{D}^\tau u_N^m:=\frac{3u_N^m-4u_N^{m-1}+u_N^{m-2}}{2\tau}.$$
 Then  we can obtain the fully discrete scheme

 \begin{equation}\label{eqn:pm_bdf2_2_discrete}
    3\widetilde{U}_{\bm{k}_V}(t_m)\approx4\widetilde{U}_{\bm{k}_V}(t_{m-1})-\widetilde{U}_{\bm{k}_V}(t_{m-2}) -2\tau\sum_{\bm{k}_U\in K_N^n}\widetilde{A}_{\bm{k}_V\overset{N}{-}\bm{k}_U}(\bm{Pk}_V)^T(\bm{Pk}_U)\widetilde{U}_{\bm{k}_V}(t_m)+2\tau\widetilde{F}_{\bm{k}_V}(t_m),
\end{equation}
 which implies that
 \begin{equation}\label{eqn:pm_bdf2_2}
     (3+2\tau Q)\mathcal{U}^m\approx4\mathcal{U}^{m-1}-\mathcal{U}^{m-2}+2\tau\mathcal{F}^m,
 \end{equation}
 where
$$
\mathcal{U}^m(i)=\widetilde{U}_{\bm{k}_V}(t_m),\quad \mathcal{U}^m(i)=\widetilde{U}_{\bm{k}_V}(t_{m-1}),\quad \mathcal{U}^m(i)=\widetilde{U}_{\bm{k}_V}(t_{m-2}),\quad \mathcal{F}^m(i)=\widetilde{F}_{\bm{k}_V}(t_m).
$$

\noindent\textbf{Computational complexity analysis.}   In the implementation of PM-BDF2-Step 1 and 2,  the availability of FFT allows
 us to compute the discrete Fourier coefficients 
 requiring $\mathcal{O}(D\log D)$ operators. Moreover, in each step of implementation, there exists the matrix-vector multiplication with the dimensions of the matrix and vector being $D\times D$ and $D$, respectively. Using the compressed matrix-vector multiplication presented in \cite{jiang2024projection}, the computational complexity of  the matrix-vector multiplication reduces to $\mathcal{O}(g\times D)$, where $g$ denotes the non-zero discrete Fourier coefficients of $A$. In conclusion, the computational complexity of PM-BDF2 in solving QPE \eqref{eqn:parabolic} is $\mathcal{O}(g\times D)$.

\section{Convergence analysis}\label{sec:convergence_analysis}

In this section, we present the convergence analysis of PM-BDF2 to solve the QPE \eqref{eqn:parabolic}.
\subsection{Main theorem}
\begin{theorem}\label{theory:error_estimate_parabolic}
Let $s\geq 1$. Suppose  $u\in L^2(0,T;\overline{H}_{QP}^{1}(\mathbb{R}^d))$ and its parent function $U\in  L^2(0,T;\overline{H}^{s+1}(\mathbb{T}^n))$. Additionally,  assume that $f\in L^2(0,T;L^{2}_{QP}(\bbR^d))$, and $F\in L^2(0,T;H^{s}(\mathbb{T}^n))$.  If the following estimate holds:
\begin{equation}\label{eqn:regularity}
    \|F(t)\|_{H^s(\bbT^n)}\leq C\|U(t)\|_{H^{s}(\bbT^n)},
\end{equation}
then we have the convergence result
$$
\|u_N^m-u(\cdot, t_m)\|_0\leq C(N^{-s}+\tau^2),
$$
where $u_N^m$ is the solution at time $t_m=m\tau$ for the  fully discrete scheme, and $u(\cdot, t_m)$ is the exact solution at $t_m$.

\end{theorem}

\noindent\textbf{Sketch of proof.}  We organize the proof of Theorem \ref{theory:error_estimate_parabolic} as follows. In Section \ref{sec:convergence_stability}, we establish the stability for $U(t)$. Then we give the error estimiate of the semi-discrete problem and  BDF2 scheme in Section \ref{sec:convergence_semi} and Section \ref{sec:convergence_bdf2}, respectively. Finally, we complete the proof of the main result in Section \ref{sec:convergence_proof}.


\subsection{Preliminary  lemmas}\label{sec:convergence_notions}
 Before proving the main conclusion, we  present some necessary  results. 
The following lemma provides the truncation error estimate and interpolation error of the PM, as also presented in  \cite[Theorem 5.1]{jiang2024numerical} and \cite[Theorem 5.3]{jiang2024numerical}, respectively.
 \begin{lemma}\label{lemma:truncation_interpolation}
 Suppose that $u(\bm{x})\in\mathrm{QP}(\mathbb{R}^d)$ and its parent function $U(\bm{y})\in {H}^s(\mathbb{T}^n)$ with $s\geq0$. There exists a constant $C$, independent of $U$ and $N$, such that
\begin{align}
\label{eqn:error_truncation}
    \|T_Nu-u\|_{0}\leq CN^{-s}|{U}|_{{H}^s(\mathbb{T}^n)},
    \\
\label{eqn:error_interpolation}
     \|I_Nu-u\|_0\leq CN^{-s}|{U}|_{{H}^s(\mathbb{T}^n)}.
\end{align}
\end{lemma}

We then estimate the truncation and interpolation errors of the gradient of quasiperiodic functions, as established in \cite[Proposition 3.3]{jiang2024projection}.  
\begin{lemma}\label{lemma:nabla}
Let $s\geq 1$. Then there exists a constant $C$ such that if $u(\bm{x})\in {H}^1_{QP}(\bbR^d)$ and parent function $U(\bm{y})\in {H}^{s+1}(\bbT^n)$, we have

\begin{equation}\label{eqn:nabla1}
    \|\alpha\nabla u - T_N(\alpha\nabla u)\|_{0}\leq CN^{-s}\|U\|_{H^{s+1}(\bbT^n)},
\end{equation}
\begin{equation}\label{eqn:nabla2}
    \|\alpha\nabla u - I_N(\alpha\nabla u)\|_{0}\leq CN^{-s}\|U\|_{H^{s+1}(\bbT^n)}.
\end{equation}


\end{lemma}

Next, We refer to Grönwall's inequality, which is commonly employed in the numerical analysis of parabolic equations.
\begin{lemma}[Grönwall's inequality]\label{lemma: Gronwall}
 Suppose that a differentable function $u$ satisfies the inequality
\begin{equation}\label{eqn:Gronwall}
    u'(t)\leq \sigma u(t)+w(t),
\end{equation}
then
$$
u(t)\leq u(0)e^{\sigma t}+\int_0^t w(\xi)e^{\sigma(t-\xi)}d\xi.
$$

\end{lemma}

For a given periodic function $U=\sum\limits_{\bm{k}\in\mathbb{Z}^n}\hat{U}_{\bm{k}}e^{\imath\bm{k}\cdot\bm{x}}$,  the $s$-order pseudo-differential operator $D^{s}:H^s(\mathbb{T}^n)\to L^2(\mathbb{T}^n)$ is defined by 
$$
D^s U := \sum_{\bm{k}\in\mathbb{Z}^n} (1+|\bm{k}|^2)^{\frac{s}{2}}\hat{U}_{\bm{k}}e^{\imath\bm{k}\cdot\bm{x}},
$$
and it satisfies the following identity:
$$
\|U\|_{H^s(\mathbb{T}^n)}=\|D^s U\|_{L^2(\mathbb{T}^n)}.
$$

To delve into the interactions between operators $D^s$ and $\mathcal{A}$, we introduce the following lemma, which is also established in \cite[pp.29]{alinhac2007pseudo}:

\begin{lemma}\label{lemma:commutator}
The order of  commutator operator $[D^s, \mathcal{A}]$ is  $s+1$. Furthermore, there exists a constant $C$ such that
$$
\|[D^s, \mathcal{A}] U\|_{L^2(\bbT^n)}\leq C\|U\|_{H^{s+1}(\bbT^n)},
$$
where $[D^s, \mathcal{A}]:= D^s\mathcal{A}-\mathcal{A}D^s$.
\end{lemma}


\subsection{Priori estimate for $U(t)$ }\label{sec:convergence_stability}
In this subsection, we we aim to demonstrate the solution 
$U(t)$ of \eqref{eqn:parabolic2} is stable with respect to  initial condition. $U(t)$  is the solution to the corresponding $n$-dimensional periodic auxiliary system derived from  \eqref{eqn:parabolic}, and is defined as 
\begin{equation}\label{eqn:parabolic2}
\left\{\begin{aligned}
     &\frac{\partial U(\bm{y}, t)}{\partial{t}}+\mathcal{A}U(\bm{y}, t)=F(\bm{y}, t), \quad (\bm{y}, t)\in\mathbb{T}^n\times (0,T)\\
     & U(\bm{y}, 0)=U^0(\bm{y}).
\end{aligned}
\right.
\end{equation}
Here, 
$$\mathcal{A}:=-\widetilde{\nabla}\cdot(A\widetilde{\nabla})=-\left(\sum_{i=1}^d\sum_{j,l=1}^n p_{ij}p_{il}\frac{\partial A}{\partial y_{j}}\frac{\partial}{\partial y_l}\right)-A\left(\sum_{i=1}^d\sum_{j,l=1}^n p_{ij}p_{il}\frac{\partial^2}{\partial y_{j}\partial y_l}\right)
$$ 
is a second order bounded operator in $n$-dimensional space and directional derivative $\widetilde{\nabla}$ is given by
$$\widetilde{\nabla}=\left(\sum\limits_{i=1}^d p_{i1}\frac{\partial}{\partial y_{1}}, \cdots, \sum\limits_{i=1}^d p_{in}\frac{\partial}{\partial y_{n}}\right).$$ 

 Building on the periodic auxiliary system \eqref{eqn:parabolic2} and Lemma \ref{lemma:commutator},  we provide the following priori estimate for the  solution $U(t)$.
\begin{theorem}\label{theorem:stability}
Let $T>0$ and $s\geq 0$ given. Under the assumptions of Theorem \ref{theory:error_estimate_parabolic},  the following inequality holds for all $t\in(0,T)$:
$$
\|U(t)\|_{H^{s}(\bbT^n)}\leq C\|U^0\|_{H^{s}(\bbT^n)}.
$$

\begin{proof}
\begin{equation*}
    \begin{aligned}
    \frac{d}{dt}\|U(t)\|_{H^s(\mathbb{T}^n)}^2 &=\frac{d}{dt}\|D^s U(t)\|_{L^2(\bbT^n)}^2=\left( D^s \frac{\partial U}{\partial t }, D^s U\right)+\left(D^s U, D^s \frac{\partial U}{\partial t }\right)\\
    &=-\left(D^s \mathcal{A}U, D^s U\right)-\left( D^s U, D^s \mathcal{A}U\right)+(D^sF, D^sU)+(D^sU, D^sF)\\
    &=-2(\mathcal{A}D^s U, D^s U)-([D^s, \mathcal{A}]U, D^s U)- (D^s U, [D^s,\mathcal{A}]U)+2(D^sF, D^sU).
   \end{aligned}
 \end{equation*}
Recall that $D^s$ is a pseudo-differential operator of order $s$. We can then combine Lemma \ref{lemma:commutator} with the fact $U\in\overline{H}^{s+1}(\mathbb{T}^n)$ to obtain that
$$
\|[D^s, \mathcal{A}]U\|_{L^2(\bbT^n)}\leq C_1\|U\|_{H^{s+1}(\bbT^n)}=C_1\|D^sU\|_{H^1(\bbT^n)}\leq C_2\|D^sU\|_{L^2(\bbT^n)}.
$$
Now, combining this with the boundedness of $\mathcal{A}$, we have
$$
\frac{d}{dt}\|U(t)\|_{H^{s}(\bbT^n)}^2\leq 2\gamma_1\|D^sU\|_{L^{2}(\bbT^n)}^2 + 2C_2\|D^sU\|^2_{L^{2}(\bbT^n)}+2\|D^s U\|_{L^2(\bbT^n)}\|D^sF\|_{L^2(\bbT^n)}.
$$
Furthermore, using the assumption \eqref{eqn:regularity}, we arrive at
$$
\frac{d}{dt}\|U(t)\|_{H^{s}(\bbT^n)}^2\leq C_3\|U(t)\|_{H^{s}(\bbT^n)}^2.
$$
 By applying Grönwall's inequality \eqref{eqn:Gronwall}, we obtain
$$
\|U(t)\|^2_{H^{s}(\bbT^n)}\leq C\|U^0\|_{H^{s}(\bbT^n)}^2, \quad\forall t\in(0,T),
$$
where the constant $C=e^{C_3T}$.

 \end{proof}
 \end{theorem}

\subsection{Error estimate for the semi-discrete problem}\label{sec:convergence_semi}
In this subsection, we focus on the error estimate for semi-discrete problem of QPE \eqref{eqn:parabolic}. Specifically, for each fixed $t>0$, we  seek  $u_N(t)\in \overline{V}^{N}$ such that it satisfies
\begin{equation}\label{eqn:parabolic_semi}
   \left\{ \begin{aligned}
        &u_N +\mathcal{L}_Nu_N=f_N\\
        &u_N^0=u(\cdot, 0),
    \end{aligned}
    \right.
\end{equation}
where the  operator $\mathcal{L}_N$  is defined by
$$
\mathcal{L}_Nu:=-\mathrm{div}\left(I_N(\alpha\nabla u)\right).
$$
With the preparations in Section \ref{sec:convergence_notions} and Section \ref{sec:convergence_stability}, we now establish the error estimate for problem \eqref{eqn:parabolic_semi} as follows.

\begin{theorem}\label{theorem:space}
Let $s\geq0$ and $T> 0$. If $u\in \overline{H}^1_{QP}(\mathbb{R}^d)$, with its parent function $U\in \overline{H}^{s+1}(\mathbb{T}^n)$, and the regularity assumption \eqref{eqn:regularity} holds, we have error estimate
$$
\|u(t)-u_N(t)\|_0\leq C N^{-s}\|U^0\|_{H^{s+1}(\bbT^n)}
$$
for all $t\in [0,T]$.
\end{theorem}

\begin{proof}
Since
\begin{align*}
  \|u(t)-u_N(t)\|_0&\leq \|u(t)-u_T( t)\|_0+\|u_T( t)- u_N(t)\|_0\\
  &:=\|\psi(t)\|_0+\|w_N(t)\|_0.  
\end{align*}
To proceed, we will estimate the bounds for $\|\psi(t)\|_0$ and $\|w_N(t)\|_0$, respectively.

First, by applying Lemma \ref{lemma:truncation_interpolation}, we have
$$
\|\psi(t)\|_0=\int_0^t \|u(\xi)-u_T( \xi)\|_0d\xi\leq C N^{-s}\|U(t)\|_{H^{s}(\mathbb{T}^n)}.
$$
Next, we  focus on the bounds for $\|w_N\|_0$. We rewrite \eqref{eqn:parabolic_semi} as follows
$$
\frac{\partial w_N}{\partial t}+\mathcal{L}_Nw_N=(\mathcal{L}_N-\mathcal{L})u_T + \frac{\partial \psi}{\partial t}+\mathcal{L}\psi+f-I_Nf.
$$
 Multipilying this relationship  by $w_N$ in terms of inner product $(\cdot, \cdot)$, we find that
$$
\left(\frac{\partial w_N}{\partial t},w_N\right)+\left(\mathcal{L}_Nw_N, w_N\right)= \left((\mathcal{L}_N-\mathcal{L})u_T,w_N\right)+\left(\frac{\partial \psi}{\partial t},w_N\right)+\left(\mathcal{L}\psi,w_N\right) +(f-I_Nf,w_N). 
$$
Combining this with the ellipicity of $\mathcal{L}_N$, we have
$$
\frac{1}{2}\frac{d}{dt}\|w_N\|_0^2+\gamma_0\|\nabla w_N\|_0^2\leq Z_1+Z_2+Z_3+Z_4,
$$
where 
$$
Z_1=\left((\mathcal{L}_N-\mathcal{L})u_T,w_N\right),
$$
$$
Z_2=\left(\frac{\partial \psi}{\partial t},w_N\right),
$$
$$
Z_3=\left(\mathcal{L}\psi,w_N\right),
$$
$$
 Z_4=(f-I_Nf, w_N).
 $$

(i) We first establish an upper bound for $Z_1$:
\begin{equation*}
	\begin{aligned}
		\left((\mathcal{L}_N-\mathcal{L})u_T,w_N\right)&=\left(-\mathrm{div}(I_N(\alpha\nabla u_T)),w_N\right)-\left(-\mathrm{div}(\alpha\nabla u_T),w_N\right)\\
		&=\left(I_N(\alpha\nabla u_T)-\alpha\nabla u_T,\nabla w_N\right)\\
		&\leq\frac{1}{2\gamma_0}\|I_N(\alpha\nabla u_T)-\alpha\nabla u_T\|_0^2+\frac{\gamma_0}{2}\|\nabla w_N\|_0^2\\
		&\leq\frac{1}{2\gamma_0}(\int_0^t\|I_N(\alpha\nabla u_T(\xi))-\alpha\nabla u_T(\xi)\|_0d\xi)^2+\frac{\gamma_0}{2}\|\nabla w_N\|_0^2\\
		&\leq CN^{-2s}\|U(t)\|_{H^{s+1}(\bbT^n)}^2+\frac{\gamma_0}{2}\|\nabla w_N\|_0^2.
	\end{aligned}
\end{equation*}

(ii) For $Z_2$, we have
\begin{equation*}
		Z_2=\left(\frac{\partial \psi}{\partial t},w_N\right)\leq \frac{1}{2\theta}\|\frac{\partial \psi}{\partial t}\|_0^2+\frac{\theta}{2}\|w_N\|_0^{2},
\end{equation*}
with
$$
\|\frac{\partial \psi}{\partial t}\|_0^2=\|\frac{\partial u}{\partial t}-T_N\Big(\frac{\partial u}{\partial t}\Big)\|_0^2\leq CN^{-2s}\|\frac{\partial U}{\partial t}\|_{H^{s}(\mathbb{T}^n)}^2.
$$
Furthermore, using \eqref{eqn:parabolic2} and the boundedness of $\mathcal{A}$,  we conclude that
$$
\|\frac{\partial U}{\partial t}\|_{H^s(\mathbb{T}^n)}\leq\|\mathcal{A}U\|_{H^s(\mathbb{T}^n)}+\|F\|_{H^s(\mathbb{T}^n)}\leq C\|U(t)\|_{H^{s}(\mathbb{T}^n)}+\|F(t)\|_{H^s(\mathbb{T}^n)},
$$
leading to
$$
Z_2\leq CN^{-2s}\|U(t)\|_{H^{s+1}(\mathbb{T}^n)}^2+\frac{\theta}{2}\|w_N\|_0^{2},
$$
where we use the assumption \eqref{eqn:regularity}.

(iii) Moving into $Z_3$, we have
\begin{equation}
		Z_3=\left(\mathcal{L}\psi,w_N\right)=\left(\alpha\nabla \psi,\nabla w_N\right).
\end{equation}
Applying the Cauchy-Schwarz inequality, we obtain
\begin{equation}
    Z_3\leq\frac{1}{2\gamma_0}\|\alpha\nabla \psi\|_0^2+\frac{\gamma_0}{2}\|\nabla w_N\|_0^2,
\end{equation}
where we have applied Lemma \ref{lemma:nabla}. Consequently,  it follows that
\begin{equation}\label{eqn:Z_3_2}
    \|\alpha\nabla \psi\|_0^2\leq CN^{-2s}\|U\|^2_{H^{s+1}(\mathbb{T}^n)}.
\end{equation}
Substituting \eqref{eqn:Z_3_2} into the expression for $Z_3$, we arrive at
$$
Z_3\leq CN^{-2s}\|U(t)\|_{H^{s+1}(\mathbb{T}^n)}^2+\frac{\gamma_0}{2}\|\nabla w_N\|_0^2.
$$

(iv) Upper bound estimate for $Z_4$ holds 
 \begin{equation*}
     \begin{aligned}
       Z_4&\leq \|f-I_Nf\|_0^2+\|w_N\|_0^2\\
       &=(\int_0^t\|f(\xi)-I_Nf(\xi)\|_0d\xi)^2+\|w_N\|_0^2\\
       &\leq CN^{-2s}|F|_{H^{s}(\mathbb{T}^n)}^2+\|w_N\|_0^2\\
       &\leq  CN^{-2s}\|F\|_{H^{s}(\mathbb{T}^n)}^2+\|w_N\|_0^2.
     \end{aligned}
 \end{equation*}
 
Gathering the terms $Z_1,Z_2,Z_3, Z_4$ we find: 
\begin{equation*}
     \begin{aligned}
      \frac{1}{2}\frac{d}{dt}\|w_N\|_0^2 &\leq Z_1+Z_2+Z_3+Z_4\\
      &\leq (\frac{\theta}{2}+1)\|w_N\|_0^2+CN^{-2s}\|U\|_{H^{s+1}(\mathbb{T}^n)}^2.
     \end{aligned}
 \end{equation*}
Combining with
Lemma \ref{lemma: Gronwall}, we obtain
$$
\|w_N(t)\|_0^2\leq e^{(\frac{\theta}{2}+1)t}CN^{-2s}\left(\|F^0\|_{H^{s}(\mathbb{T}^n)}^2+\int_{0}^{t}\|U(\xi)\|_{H^{s+1}(\mathbb{T}^n)}^2 d\xi\right),
$$
where we use the assumption \eqref{eqn:regularity}.
According to Theorem \ref{theorem:stability}, we conclude that 
$$
\|U(t)\|_{H^{s+1}(\mathbb{T}^n)}\leq C\|U^0\|_{H^{s+1}(\mathbb{T}^n)}.
$$
Consequently, we derive that
$$
\|w_N(t)\|_0\leq CN^{-s}(\|F^0\|_{H^{s}(\mathbb{T}^n)}^2+\|U^0\|_{H^{s+1}(\mathbb{T}^n)}^2)^{\frac{1}{2}}.
$$

In conclusion, we obtain the error analysis for semi-discrete problem of QPE \eqref{eqn:parabolic}
\begin{align*}
    \|u(t)-u_N(t)\|_0&\leq\|\psi(t)\|_0+\|w_N(t)\|_0\\
    &\leq CN^{-s}(\|F^0\|_{H^{s}(\mathbb{T}^n)}^2+\|U^0\|_{H^{s+1}(\mathbb{T}^n)}^2)^{\frac{1}{2}}\\
    &\leq CN^{-s}\|U^0\|_{H^{s+1}(\mathbb{T}^n)}.
\end{align*}

\end{proof}
\subsection{Error estimate for BDF2 scheme}\label{sec:convergence_bdf2}
{In this subsection, we provide an error estimate for the BDF2 scheme in the time direction.
It is well known that  the BDF2 scheme achieves second-order accuracy in the time step $\tau$, resulting in a local truncation error of $\mathcal{O}(\tau^3)$. Building on this, the following lemma establishes the solution accuracy of the BDF2 scheme is of order $\mathcal{O}(\tau^2)$, a result also presented in \cite[Theorem 10.7]{thomee2007galerkin}.}
\begin{lemma}\label{lemma:time_solution_accuracy}
    Let $u^m$ be the solution of the semi-discrete problem at $t_m$ and $u(\cdot, t_m)$ the solution of the problem \eqref{eqn:parabolic}. Then we have the following error estimate
    \begin{equation}\label{eqn:bdf2_error_numerical}
      \|u^m-u(\cdot, t_m)\|_0\leq C_T\tau^2\|U^0\|_{H^{s+1}(\mathbb{T}^n)},
    \end{equation}
    where $C_T$ is a constant that depends on  $T$.
\end{lemma}
\begin{proof}
Since the local truncation error of BDF2 scheme is of order $\mathcal{O}(\tau^3)$, we can subsequently derive the following result
\begin{equation}
    \|u^m-u(\cdot, t_m)\|_0\leq \sum_{j=0}^{m-1}\|u^j-u(\cdot, t_j)\|_0\leq C_1m\tau^3\leq C_1T\tau^2=C_T\tau^2,
\end{equation}
where $T=M\tau$.
\end{proof}

\subsection{Proof of Theorem \ref{theory:error_estimate_parabolic}}\label{sec:convergence_proof}
Using the Theorem \ref{theorem:space} and Lemma \ref{lemma:time_solution_accuracy},   we are now in a position to provide the  error estimate for  fully discrete scheme \eqref{eqn:full_discrete}.
\begin{proof}
 From Theorem \ref{theorem:space}, we have
$$
\|u_N^m-u^m\|_0\leq\sum_{j=0}^{m-1}\|u_N^j-u^j\|_0\leq C_0 m N^{-s}\|U^0\|_{{H^{s+1}(\mathbb{T}^n)}}^2\leq CN^{-s}.
$$
Using Lemma \ref{lemma:time_solution_accuracy} , we deduce that
$$
\|u^m-u(\cdot, t_m)\|_0\leq C_T\tau^2,
$$
where $T=M\tau$. Applying the triangle inequality, we obtain 
$$
\|u_N^m-u(\cdot,t_m)\|_0\leq \|u_N^m-u^m\|_0 + \|u^m-u(\cdot,t_m)\|_0,
$$
which leads to the result
$$
\|u_N^m-u(\cdot, t_m)\|_0\leq C_T(N^{-s}+\tau^2).
$$
\end{proof}

\section{Numerical experiments}\label{sec:numerical_experiments}
In this section, we conduct some numerical experiments to demonstrate the performance of PM-BDF2 for solving QPE \eqref{eqn:parabolic}.  These experiments  are conducted using MATLAB R2023a on a laptop computer equipped with an Intel Core 2.20GHz CPU and 16GB RAM. The computational time, measured in seconds(s), is referred to as CPU time.

We use $L_{QP}^2(\bbR^d)$-norm to measure the numerical error at final time
$$\mathrm{Err}_h^\tau=\|u_N^M-u(\cdot, t_M)\|_{L^2_{QP}(\bbR^d)},$$
where $h=2\pi/N$ is mesh size of $n$-dimensional tours and $\tau$ denotes time step size. The error order in the time direction is calculated by 
$$
\kappa=\frac{\ln \left(\operatorname{Err}_h^{\tau_{1}}/\operatorname{Err}_h^{\tau_{2}}\right)}{\ln \left(\tau_{1} / \tau_{2}\right)}.
$$

 Our numerical examples focus on verifying  the accuracy in space direction and the
 error order in time direction. Consequently, the final time 
$T$ can be chosen arbitrarily.
 
 \subsection{One-dimensional case}
Consider the QPE \eqref{eqn:parabolic} with the  quasiperiodic coefficient $$\alpha(x)=\cos (2 \pi x)+\cos (2 \sqrt{5} \pi x)+6,$$ and the initial condition 
 $$\sum_{\lambda \in \Lambda_{L}} \widetilde{u}_{\lambda} e^{\imath 2\pi \lambda x},\quad \widetilde{u}_{\lambda}=e^{-(|m|+|n|)},$$  where  $$\Lambda_{L}=\{\lambda=m+n \sqrt{5},-16 \leq m, n \leq 15\}.$$
 The exact solution is 
 $$u=\sum_{\lambda \in \Lambda_{L}} e^{-\imath t}\left(\widetilde{u}_{\lambda} e^{\imath2\pi \lambda x}\right),\quad \widetilde{u}_{\lambda}=e^{-(|m|+|n|)},$$  where  $$\Lambda_{L}=\{\lambda=m+n \sqrt{5},-16 \leq m, n \leq 15\}.$$ 

We begin by verifying the spectral convergence of PM. To minimize the impact of temporal discretization on spatial accuracy of PM, we set $ T=10^{-4}, M=10^3, \tau=T/M= 10^{-7}$. The numerical results are as follows. In Table \ref{tab:parabolic:two_modes:combination_pm}, we observe the
spectral accuracy convergence of $\mathrm{Err}_h^\tau$ as $N$ increases sufficiently to cover all non-zero values
of $\widetilde{u}_{\lambda}$, which aligns with our convergence result in Theorem \ref{theory:error_estimate_parabolic}. 

While verifying the second-order error precision of applying BDF2 scheme to solve QPE \eqref{eqn:parabolic} in time direction, we adopt $h=2\pi/16$ to ensure that the approximation error in the spatial direction remains insignificant compared to the error in the temporal direction. The results presented in  Table \ref{tab:parabolic:two_modes:combination_bdf2} demonstrate that the BDF2 scheme maintains a second-order 
accuracy in temporal computations.
 \begin{table}[!hbpt]
 \vspace{-0.2cm}
  \centering
 \footnotesize{
 \caption{When  $\tau=1 \times 10^{-7}$, numerical error and CPU time of PM-BDF2 for different $N$. } \label{tab:parabolic:two_modes:combination_pm}
	\begin{tabular}{|c|c|c|c|c|c|}
\hline 
$N$  & 4 & 8 & 16 & 32 & 64 \\
\hline 
$\mathrm{Err}_h^\tau$ &  6.713e-03  &  1.310e-04  &  2.956e-06  &  3.742e-13  &  3.799e-13  \\
\hline CPU time(s) & 1.125e-02 & 7.831e-02 & 3.251e-01 & 1.299 &  5.302  \\
\hline
\end{tabular}
 }
\end{table}

\vspace{-10pt}
\begin{table}[!hbpt]
\vspace{-0.2cm}
\centering
 \footnotesize{
 \caption{Temporal error and error order of PM-BDF2 with  $N=32$. }\label{tab:parabolic:two_modes:combination_bdf2}
	\begin{tabular}{|c|c|c|c|c|}
\hline 
$\tau$  &  $1 \times 10^{-5}$  &  $5 \times 10^{-6}$  &  $2.5 \times 10^{-6}$  &  $1.25 \times 10^{-6}$  \\
\hline 
$\mathrm{Err}_h^\tau$ &  3.882e-09  &  9.705e-10  &  2.426e-10  &  6.065e-10  \\
\hline
$\kappa$  & - & 2.00 & 2.00 & 2.00 \\
\hline
\end{tabular}
 }
\end{table}

\subsection{Two-dimensional cases}
In this subsection, we further demonstrate that the PM is a highly precise and efficient
algorithm to solve QPE through two-dimensional examples. Consider the QPE \eqref{eqn:parabolic} with coefficient 
$$
\alpha(x,y)=\cos (2 \pi x)+\cos (2 \sqrt{5} \pi x)+\cos (2 \pi y)+12,\quad (x,y)\in\bbR^2,
$$
and the  initial value
$$
u^0(x, y)=e^{\imath 2 \pi x}+e^{\imath 2 \sqrt{5} \pi x}+e^{\imath 2 \pi y}.
$$
The corresponding projection matrix is 

$$
\bm{P}=2\pi\begin{bmatrix}
    	1 & \hspace{2mm}\sqrt{5} &\hspace{2mm}0\\
    0 &\hspace{2mm}0&\hspace{2mm}1\\    
\end{bmatrix}.
$$
Therefore, this quasiperiodic system can be embedded into a three-dimensional parent
system.
The exact solution on $\bbR^2$ is
$$
u=e^{-\imath t}\left(e^{\imath 2 \pi x}+e^{\imath 2 \sqrt{5} \pi x}+e^{\imath 2 \pi y}\right).
$$

In Table \ref{fig:parabolic:three_modes}, we examine the dependence of errors on both  $N$ and $\tau$. Specifically, Section \ref{fig:paraboic:three_modes:space} presents numerical error $\mathrm{Err}_{h}^\tau$ for $T=0.01$ with time steps set at $\tau=10^{-4}, 10^{-6}$ and $10^{-12}$.  Regardless of the choice of $N$, it is confirmed in Figure \ref{fig:paraboic:three_modes:space} that the errors with $\tau=10^{-12}$ is more smaller than those with $\tau=10^{-4}$ or $\tau=10^{-6}$. This means  the error depends only on $\tau$ and not on $N$. Furthermore, when $\tau$ is sufficiently small, such as  $\tau=10^{-12}$, the error can reach machine error, demonstrating that the  high accuracy of PM in solving QPE. 

{For the BDF2 scheme, as shown in Figure \ref{fig:parabolic:three_modes:time}, the temporal  convergence rate of PM-BDF2 is of order $\mathcal{O}(\tau^2)$, as reflected by the slope of yellow dashed line.} Moreover, we conclude that if $N$ is sufficient large (with $N\geq 8$ being adequate in this case), the error depends only on the size of $\tau$ and not on the size of $N$. This is primarily due to the use of the PM in the spatial direction.

 \begin{figure}[!hbpt]	
       \hspace{12mm}
	\subfigure[] 
	{
		\begin{minipage}{7cm}
			\centering         
			\includegraphics[width=6cm]{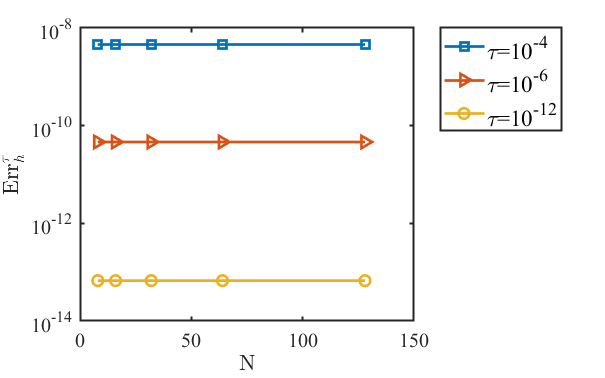}   
            \label{fig:paraboic:three_modes:space}
		\end{minipage}
 	}
        \hspace{-9mm}
	\subfigure[]
	{
		\begin{minipage}{7cm}
			\centering      
			\includegraphics[width=6cm]{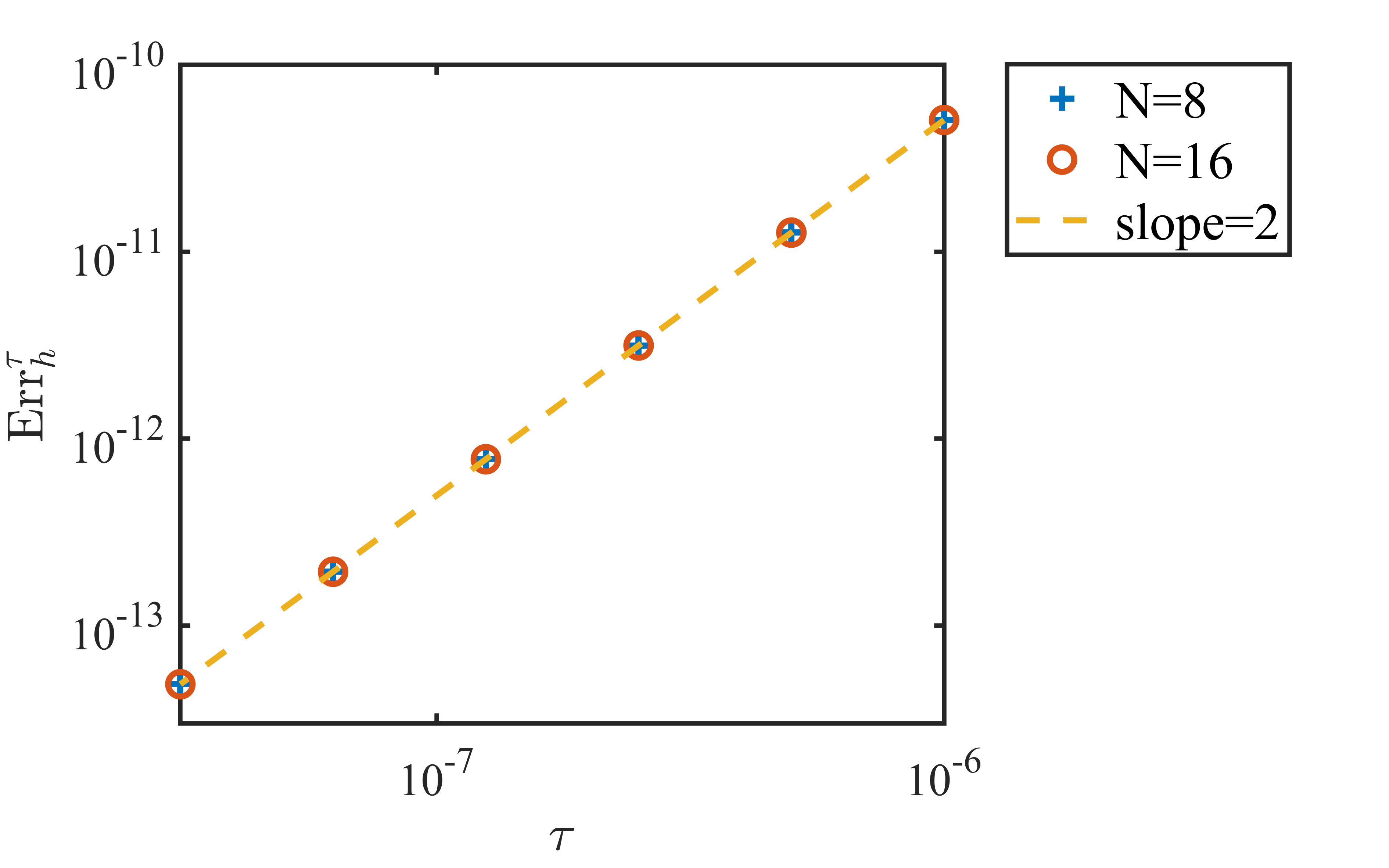}   
            \label{fig:parabolic:three_modes:time}
		\end{minipage}
	}
	\caption{Figure \ref{fig:paraboic:three_modes:space}: The relationship between the error and the spatial  parameter $N$ for solution at $T=0.01$; Figure \ref{fig:parabolic:three_modes:time}: The relationship between the error and the time step $\tau$ for solution at $T=0.01$. } 
	\label{fig:parabolic:three_modes}  
\end{figure}

Next, we design a more complex test to demonstrate PM-BDF2 has spectral accuracy in space. We choose 
$$
\alpha(x,y)=\cos x+\cos (\sqrt{5} x)+\cos y+12,\quad (x,y)\in\bbR^2,
$$
and give the exact solution in form
$$
u(x,y)=e^{-\imath  t}\left(\sum_{\lambda \in \Lambda_{L}} \widetilde{u}_{\lambda} e^{\imath \lambda x} + e^{\imath y}\right),\quad \widetilde{u}_{\lambda}=e^{-(|m|+|n|)},
$$
where
$$\Lambda_{L}=\{\lambda=m+n \sqrt{5},-16 \leq m, n \leq 15\}.$$ The initial condition is 
$$
u^0=\sum_{\lambda \in \Lambda_{L}} e^{-\imath t}\left(\widetilde{u}_{\lambda} e^{\imath \lambda x}\right) + e^{\imath y}
$$
and the corresponding projection matrix is 
$$
\bm{P}=
\begin{bmatrix}
	1 &\hspace{2mm} \sqrt{5} &\hspace{2mm} 0\\
    0 &\hspace{2mm} 0&\hspace{2mm} 1\\
\end{bmatrix}.
$$

We  select parameters $T=10^{-5}$, $M=10^2$, and $\tau=T/M=10^{-7}$. In  Table \ref{tab:parabolic:two_modes:combination_pm_2d}, we present the numerical error $\mathrm{Err}_h^\tau$ along with CPU time of PM-BDF2. Visually, PM exhibits spectral convergence in spatial direction. When the fine mesh size is set to $h=16/\pi$, the spatial  error becomes negligible in comparison to the temporal error. As shown in Table \ref{tab:parabolic:two_modes:combination_bdf2_2d}, this allows the PM-BDF2 method to achieve second-order accuracy in time with $N=32$.
\begin{table}[!hbpt]
\vspace{-0.2cm}
\centering
 \footnotesize{
 \caption{When  $\tau=1 \times 10^{-7}$, numerical error and CPU time of PM-BDF2 for different $N$. }\label{tab:parabolic:two_modes:combination_pm_2d}
	\begin{tabular}{|c|c|c|c|c|}
\hline 
$N$  & 4 & 8 & 16 &32  \\
\hline 
$\mathrm{Err}_h^\tau$ &  1.210-03  &  2.144e-04  &  3.183e-07&4.643e-13  \\
\hline CPU time(s) & 3.607e-01 & 2.802 & 1.214e+01&9.681e+01   \\
\hline
\end{tabular}
 }
\end{table}

\begin{table}[!hbpt]
\vspace{-0.2cm}
\centering
 \footnotesize{
  \caption{Temporal error Temporal error and error order of PM-BDF2 with  $N=32$. }\label{tab:parabolic:two_modes:combination_bdf2_2d}
	\begin{tabular}{|c|c|c|c|c|}
\hline 
$\tau$  &  $1 \times 10^{-6}$  &  $5 \times 10^{-7}$  &  $2.5 \times 10^{-7}$  &  $1.25 \times 10^{-7}$  \\
\hline 
$\mathrm{Err}_h^\tau$ &  4.886e-11  &  1.221e-11  &  3.045e-12  &  7.516e-13  \\
\hline
$\kappa$  & - & 2.00 & 2.00 & 2.00 \\
\hline
\end{tabular}
 }
\end{table}


\section{Conclusions}\label{sec:conclusion}
In this paper, we develop a high-accuracy numerical method called PM-BDF2 for solving arbitrary-dimensional QPE \eqref{eqn:parabolic}. A rigorous convergence analysis shows that PM-BDF2 achieves spectral accuracy in space, provided the exact solution has sufficient regularity, and second-order accuracy in time. Finally, our theoretical findings are supported by one- and two-dimensional numerical experiments.


\begin{thebibliography}{100}
\bibitem{amann1978periodic}
{H. Amann}, {\it Periodic solutions of semilinear parabolic equations},  Nonlinear analysis, (1978), 1--29.

\bibitem{alinhac2007pseudo}
{S. Alinhac and P. Gérard}, {\it Pseudo-differential operators and the Nash-Moser theorem }, vol.82, American Mathematical Soc., 2007.

\bibitem{braack2011duality}
{M. Braack and E. Burman}, {\it Duality based a posteriori error estimation for quasi-periodic solutions using time averages}, SIAM Journal on Scientific Computing, {\bf33}(2011), 2199--2216.

\bibitem{bastidas2021numerical}
{M. Bastidas and C. Bringedal}, {\it  Numerical homogenization of non-linear parabolic problems on adaptive meshes}, Journal of Computational Physics, (2021), 425, 109903.

\bibitem{cao2021computing}
{D. Cao, J. Shen and J. Xu}, {\it Computing interface with quasiperiodicity}, Journal of Computational Physics, {\bf424}(2021), 109863.


\bibitem{douglas1956numerical}
{J. Douglas and H.H. Rachford}, {\it On the numerical solution of heat conduction problems in two and three space variables}, Transactions of the American mathematical Society, {\bf82}(1956), 421--439.

\bibitem{evans2022partial}
{L. C. Evans}, {\it Partial differential equations},  vol.19, American Mathematical Society(2022).

\bibitem{franzoi2024kam}
{L. Franzoi and R.Montalto}, {\it  A KAM Approach to the Inviscid Limit for the 2D Navier–Stokes Equations}, In Annales Henri Poincaré, (2024), 1--45.

\bibitem{jiang2014numerical}
{K. Jiang and P. Zhang}, {\it Numerical methods for quasicrystals}, Journal of Computational Physics, {\bf256}(2014), 428--440.

\bibitem{jiang2015stability}
{K. Jiang, J. Tong, P. Zhang and A. C. Shi}, {\it Stability of two-dimensional soft quasicrystals in systems with two length scales}, Physical Review E, {\bf92}(2015), 042159.

\bibitem{jiang2018numerical}
{K. Jiang and P. Zhang}, {\it Numerical mathematics of quasicrystals}, Proceedings of the International Congress of Mathematicians: Rio de Janeiro 2018, (2018), 3591--3609.

\bibitem{jiang2022tilt}
{K. Jiang, W. Si and J. Xu}, {\it Tilt grain boundaries of hexagonal structures: a spectral viewpoint}, SIAM Journal on Applied Mathematics, {\bf82}(2022), 1267--1286.

\bibitem{jiang2023approximation}
{K. Jiang, S. Li and P. Zhang}, {\it On the approximation of quasiperiodic functions with Diophantine frequencies by periodic functions}, accepted by SIAM Journal on Mathematical Analysis, also see arXiv:2304.04334.

\bibitem{jiang2024numerical}
{K. Jiang, S. Li and P. Zhang}, {\it Numerical methods and analysis of computing quasiperiodic systems}, SIAM Journal on Numerical Analysis, {\bf62}(2024), 353--375.

\bibitem{jiang2024high}
{K. Jiang, S. Li, and J. Zhang}, {\it High-accuracy numerical methods and convergence analysis for Schr\" odinger equation with incommensurate potentials}, Journal of Scientific Computing,  {\bf108}(2024), 18.

\bibitem{jiang2024projection}
{K. Jiang, M. Li, J. Zhang and L. Zhang}, {\it Projection method for quasiperiodic elliptic equations and application to quasiperiodic homogenization}, arXiv preprint arXiv:2404.06841, (2024).

\bibitem{eriksson1991adaptive}
{E. Kenneth, and J. Claes}, {\it Adaptive finite element methods for parabolic problems I: A linear model problem}, SIAM Journal on Numerical Analysis, {\bf28}(1991), 43--47.

\bibitem{xueyang2021numerical}
{X. Li, and K. Jiang}, {\it Numerical simulation for quasiperiodic quantum dynamical systems}, Journal on Numerica Methods and Computer Applications, {\bf42}(2021), 3.

\bibitem{maestrello1979quasi}
{L. Maestrello, Y. Fung}, {\it Quasi-periodic structure of a turbulent jet},  Journal of Fluid Mechanics, {\bf64}(1979), 107--122.

\bibitem{mercier1989introduction}
{B. Mercier}, {\it An introduction to the numerical analysis of spectral methods},  Berlin: Springer-Verlag, 1989.

\bibitem{montalto2021navier}
{R. Montalto}, {\it The Navier–Stokes equation with time quasi-periodic external force: existence and stability of quasi-periodic solutions}, Journal of Dynamics and Differential Equations, {\bf33}(2021), 1341--1362.

\bibitem{nakao1976bounded}
{M. Nakao and T. Nanbu}, {\it Bounded or almost periodic classical solutions for some nonlinear parabolic equations}, Memoirs of the Faculty of Science, Kyushu University. Series A, Mathematics, {\bf30}(1976), 191--211.

\bibitem{pitt1942some}
{H. Pitt}, {\it Some generalizations of the ergodic theorem}, Mathematical Proceedings of the Cambridge Philosophical Society, {\bf38}(1942), 325--343.

\bibitem{nakao1978bounded}
{M. Nakao}, {\it Bounded, periodic and almost periodic classical solutions of some nonlinear wave equations with a dissipative term}, Journal of the Mathematical Society of Japan, {\bf30}(1978), 375--394.

\bibitem{strikwerda2004finite}
{J. C. Strikwerda}, {\it Finite difference schemes and partial differential equations},  SIAM(2004).

\bibitem{takeda1999quasi}
{Y. Takeda}, {\it Quasi-periodic state and transition to turbulence in a rotating Couette system},  Journal of Fluid Mechanics, {\bf389}(1999), 81--89.

\bibitem{thomee2007galerkin}
{V. Thom{\'e}e}, {\it Galerkin finite element methods for parabolic problems},  vol.25, Springer Science \& Business Media(2007).

\bibitem{ward1988bounded}
{J. R. Ward}, {\it Bounded and almost periodic solutions of semi-linear parabolic equations},  The Rocky Mountain journal of mathematics, {\bf18}(1988), 479--494.

\bibitem{yoshizawa2012stability}
{T. Yoshizawa}, {\it Stability theory and the existence of periodic solutions and almost periodic solutions},  vol.14, Springer Science \& Business Media(2012).

\bibitem{zaidman1961soluzioni}
{S. Zaidman}, {\it Soluzioni limitate e quasi-periodiche dell'equazione del colore non-omogenea. I}, Rend. Accad. Naz. Lincei, (1961).

\bibitem{zhikov2006estimates}
{V. V. Zhikov}, {\it Estimates of homogenization for a parabolic equation with periodic coefficients},  Russian Journal of Mathematical Physics, {\bf13}(2006), 224--237.

\end{thebibliography}
\end{document}